\documentclass[11pt,reqno]{amsart}
\usepackage{hyperref}
\usepackage[top=3cm, bottom=3cm, left=3cm, right=3cm] {geometry}

\usepackage{amsmath, amssymb, latexsym, amscd, amsthm, amsfonts, amstext}
\usepackage[mathscr]{eucal}
\usepackage{xspace}
\usepackage[dvips]{color}
\usepackage{xcolor}

\theoremstyle{plain}
\newtheorem{theorem}{Theorem}[section]

\newtheorem{lemma}[theorem]{Lemma}
\newtheorem{remark}[theorem]{Remark}

\theoremstyle{definition}
\newtheorem{definition}[theorem]{Definition}

\allowdisplaybreaks
\begin{document}
\title[Singular  $p$-biharmonic problems]{Singular $p$-biharmonic problems involving the Hardy-Sobolev exponent}
\centerline{}
\author{A. Drissi} 
\address{A.  Drissi, ORCID ID: 0000-0003-2704-4997
\newline
Department of Mathematics, Faculty of Sciences, University of Tunis El Manar, 2092 Tunis, Tunisia.}
\email{\href{mailto: A. Drissi <amor.drissi@ipeiem.utm.tn>}{amor.drissi@ipeiem.utm.tn}}
\author{A.  Ghanmi}
\address{A.  Ghanmi, ORCID ID: 0000-0001-8121-0496 
 \newline
Department of Mathematics, Faculty of Sciences, University of Tunis El Manar, 2092 Tunis, Tunisia.}
\email{\href{mailto: A. Ghanmi <abdeljabbar.ghanmi@gmail.com>}{abdeljabbar.ghanmi@gmail.com}}
\author{D.D.  Repov\v{s}}
\address{D.D.  Repov\v{s}, ORCID ID: 0000-0002-6643-1271
 \newline
Faculty of Education, University of Ljubljana, 1000 Ljubljana, Slovenia.\newline
Faculty of Mathematics and Physics, University of Ljubljana, 1000 Ljubljana,
Slovenia.\newline
Institute of Mathematics, Physics and Mechanics, 1000 Ljubljana, Slovenia.}
\email{\href{mailto: D.D.  Repov\v{s} <dusan.repovs@guest.arnes.si.>}{dusan.repovs@guest.arnes.si}}
\begin{abstract}
This paper is concerned with existence results for the singular    
$p$-biharmonic problem involving the Hardy potential
and the critical Hardy-Sobolev exponent.  More precisely,  by using variational methods combined with the Mountain pass theorem and the Ekeland variational principle, we establish the existence and multiplicity of solutions. To illustrate the usefulness of our
results, an illustrative example is also presented.
\end{abstract}
\date{}
\maketitle
\numberwithin{equation}{section}
\section{Introduction}\label{s1}
Recently, a lot of attention has been paid to the study of problems involving the $p$-Laplacian operator and the $p$-biharmonic operator.  We refer the reader to {\sc Alsaedi et al.} \cite{dhifli},  {\sc Chung et al.} \cite{chung}, {\sc Huang and  Liu} \cite{hwa}, 
and
{\sc Sun and  Wu} \cite{sun2, sun}. The reasons to study these types of problems are their applications in many fields such as quantum mechanics, flame propagation, and the traveling waves in suspension bridges (for more applications see, e.g., {\sc Bucur and Valdinoci} \cite{1}
and
 {\sc  Lazer and McKenna} \cite{11}).  Problems involving Hardy terms have been extensively investigated by several authors, see, e.g., {\sc Bhakta et al.} \cite{ bhak, bhak1}, {\sc Ghoussoub and Yuan} \cite{8}, 
 and
 {\sc Guan et al.} \cite{x2}. Also, various problems involving the critical Hardy-Sobolev exponent have been widely studied, see, e.g.,
  {\sc Chaharlang and  Razani} \cite{chah}, {\sc Chen et al.} \cite{x1}, {\sc Perera and Zou} \cite{23},  {\sc P\'erez-Llanos and Primo} \cite{pere}, {\sc Wang} \cite{wang}, 
  and
  {\sc Wang and Zhao} \cite{wang2}. 

In particular, {\sc Ghoussoub and Yuan} \cite{8}  used variational methods to study the question of the existence  of solutions of  the following problem
$$\begin{cases}-\Delta_{p}\varphi=\lambda |\varphi|^{r-2}\varphi+\mu\frac{|\varphi|^{q-2}\varphi}{|x|^\alpha}&\mbox{in }\;\Omega,\\
\varphi=0&\mbox{on}\;\partial\Omega,
\end{cases}$$
where $\Omega\subset \mathbb{R}^n$ is a regular bounded domain, $\mu$ and $\lambda$ are positive parameters,  $\min(q,r)\geq p$,   $q\leq p^\ast(\alpha),$ and $ r\leq p^\ast$.

{\sc Perrera and Zou} \cite{23}  investigated  the following critical Hardy  $p$-Laplacian problem 
\begin{equation}
\label{eq2}
\begin{cases}\displaystyle
-\Delta_{p}\varphi=\lambda |\varphi|^{p-2}\varphi+\frac{|\varphi|^{p^{\ast}(\alpha)-2}\varphi}{|x|^\alpha}&\mbox{in }\;\Omega,\\
\varphi=0&\mbox{on}\;\partial\Omega.
\end{cases}\end{equation}
More precisely, they used variational methods to 
establish
 the multiplicity of solutions of problem \eqref{eq2}.
Recently, {\sc Wang} \cite{wang} considered the following  problem
\begin{equation}\label{eq3}\begin{cases}\Delta^2_{p}\varphi=h(x,\varphi)+\mu\frac{|\varphi|^{r-2}\varphi}{|x|^s}&\mbox{in }\;\Omega,\\
\varphi=\Delta \varphi=0&\mbox{on}\;\partial\Omega.
\end{cases}\end{equation}
He used the Mountain pass theorem to establish
 the existence of solutions of problem \eqref{eq3}. Moreover, the existence of multiple solutions was established
by applying the Fountain Theorem.

Motivated by the above
mentioned results, we study in the present paper the existence and the multiplicity of solutions of the following  singular $p$-biharmonic problem involving  the  Hardy potential and the critical Hardy-Sobolev exponent
\begin{equation}\label{p}
\Delta_{p}^{2}\varphi-\lambda \frac{|\varphi|^{p-2}\varphi}{|x|^{2p}}+\Delta_p \varphi=  \mu f(x)h(\varphi)+ \frac{|\varphi|^{p^\ast(\alpha)-2}\varphi}{|x|^{\alpha}}
 \quad \mbox{ in }\mathbb{R}^N,
\end{equation}
where $0\leq \alpha <2p$, $1<p<\frac{N}{2}$,  $\lambda>0$,  $\mu>0,$  and $p^\ast(\alpha):=\frac{p(N-\alpha)}{N-2 p}.$
Here,
$\Delta_{p} $ and $\Delta_{p}^{2}$ are the $p$-Laplacian  and the $p$-biharmonic operator,  respectively, defined  by
$\Delta_{p}\varphi:=\mbox{div}(|\nabla \varphi|^{p-2}\nabla \varphi)$
and
$\Delta^2_{p}\varphi:=\Delta(|\Delta \varphi|^{p-2}\Delta \varphi),$
respectively,  $f$ is a positive function,    $h$ is  a
 continuous
 function, and     the  following hypotheses are satisfied
 
$(H_1)$  There exists  $r\in(p,p^\ast)$
 such that
 $f \in L^\infty(\mathbb{R}^N)$
 and
 $$|h(\varphi)|\leq c_1 |\varphi|^{r-1}, \;\mbox{for every}\;\varphi\in E
 \
 \hbox{and some positive constant}
 \
 c_1,$$ where $p^\ast:=\frac{pN}{N-2 p}$  and the space  $E:=W^{2,p}(\mathbb{R}^N)$ is defined 
  in Section \ref{s2}.
  
$(H_2)$  There exists $\sigma>0$ such that for every $y\in \mathbb{R}^N $,  we have
$$
0<r  H(\varphi)\leq h(\varphi)\varphi, \;|\varphi|\geq\sigma>0,
\
\hbox{where}
\
H(t):=\int_{0}^{t}h(s)ds.
$$

$(H_3)$  There exist $c_1>0$, $1<r<p,$  and $s \in (\frac{p^\ast}{p^\ast-r}, \frac{p}{p-r})$
 such that
 $$
 0<f \in L^{\frac{p^\ast}{p^\ast-r}}(\mathbb{R}^N)\cap L^s_{loc}(\mathbb{R}^N)
 \
\hbox{and}
\
 |h(\varphi)|\leq c_1 |\varphi|^{r-1},\,\mbox{for every}\;\varphi\in E.
 $$

The following are the main results of this paper.
\begin{theorem}\label{thm}
Suppose that  hypotheses $(H_1)$ and $(H_2)$ hold.  
Then for every
 $\mu>0$, the  singular $p$-biharmonic problem \eqref{p}  has
  at least one   nontrivial weak solution,   provided that $\lambda>0$ is small enough. 
\end{theorem}
\begin{theorem}\label{thmm} 
Suppose that  hypotheses $(H_2)$ and $(H_3)$ hold.  
Then  there exists $\mu_0>0$  such that for every
 $\mu\in(0,\mu_0)$,  the singular $p$-biharmonic problem \eqref{p}  has
  at least two   nontrivial weak solutions,  provided that $\lambda>0$ is small enough. 
\end{theorem}

\begin{remark}
Note that   the singular $p$-biharmonic problem \eqref{p}
  is very important
  since it contains the $p$-biharmonic operator, the $p$-Laplacian operator, the singular nonlinearity,  and the Hardy potential. Moreover, it appears in many applications,  such as non-Newtonian fluids,  viscous fluids,  traveling waves in suspension bridges, and various other physical phenomena, see, e.g.,  {\sc Chen et al.} \cite{70}, {\sc Lazer and McKenna} \cite{11},
  and
   {\sc Ru\v{z}i\v{c}ka} \cite{140}. 
\end{remark}
The  paper is organized as follows: 
In Section \ref{s2},  we present some variational framework related to problem \eqref{p}.   
 In  Section \ref{s3}, we prove   Theorem \ref{thm}.   In Section \ref{s4}  we  combine the Mountain pass theorem with the Ekeland variational principle to prove the multiplicity of solutions of problem \eqref{p}
 (Theorem \ref{thmm}). 
  In Section \ref{s5} an example is presented to illustrate our main results.
  Finally, in Section \ref{s6} we summarize the main contributions of this paper.

\section{Preliminaries}\label{s2}
We begin by recalling  some necessary facts
related to the Hardy-Sobolev exponent nonlinearity. 
We finish this section by presenting the variational framework related to problem \eqref{p}.
For other necessary background material we refer to the comprehensive monograph
by {\sc Papageorgiou et al.} \cite{PRR}.

It is well-known that the Hardy-Sobolev exponent is closely related to the  following Rellich inequality (see  {\sc Davies and  Hinz} \cite[p. 520]{davi})
\begin{equation}\label{rellich}
  \int_{\mathbb{R}^N}\frac{|\varphi(x)|^p}{|x|^{2p}}\,dx\leq \left(\frac{p^2}{N(p-1)(N-2p)}\right)^p\int_{\mathbb{R}^N}|\Delta \varphi(x)|^p\,dx,\;\mbox{for every }\;\varphi\in W^{2,p}(\mathbb{R}^N),
\end{equation}
where $W^{2,p}(\mathbb{R}^N)$ denotes  the Sobolev space which is  defined by
$$W^{2,p}(\mathbb{R}^N):=\left\{\varphi\in L^p(\mathbb{R}^N): \Delta \varphi ,\;|\nabla \varphi| \in L^p(\mathbb{R}^N)\right\}.$$
For more details about this space, see,
e.g.,   {\sc Davies and  Hinz} \cite{davi}, {\sc Mitidieri} \cite{miti},
and {\sc Rellich} \cite{reli}.

According to the Rellich inequality \eqref{rellich}, $W^{2,p}(\mathbb{R}^N)$ can be endowed  with the following norm
$$\|\varphi\|:=\left(\int_{\mathbb{R}^N}|\Delta \varphi(x)|^p-\lambda \frac{|\varphi(x)|^p}{|x|^{2p}}+|\nabla \varphi(x)|^p\,dx\right)^{\frac{1}{p}},$$
provided that \begin{equation} \label{lamb}0<\lambda < \left(\frac{N(p-1)(N-2p)}{p^2}\right)^p.\end{equation}
The space $W^{2,p}(\mathbb{R}^N)$ is continuously embeddable into $L^{\sigma}(\mathbb{R}^N)$ for every  $p\leq \sigma \leq p^\ast$, and compactly embeddable into $L^{\sigma}_{loc}(\mathbb{R}^N),$ for every $p\leq \sigma < p^\ast$. Moreover, for every $\varphi\in W^{2,p}(\mathbb{R}^N)$, one has
\begin{equation}\label{const}
 |\varphi|_\sigma\leq S_\sigma^{-\frac{1}{p}}\|\varphi\|,
\end{equation}
where $|\varphi|_p$ denotes the usual $L^{p}(\mathbb{R}^N)$-norm  and $S_\sigma$ is defined  by
\begin{equation}
 \label{Slambda}
 S_\sigma:=\displaystyle\inf_{\varphi\in W^{2,p}(\mathbb{R}^N), \varphi\neq 0}\frac{\int_{\mathbb{R}^N}|\Delta \varphi(x)|^p-\lambda \frac{|\varphi(x)|^p}{|x|^{2p}}+|\nabla \varphi(x)|^p\,dx}{\left(\int_{\mathbb{R}^N}|x^{-\alpha}||\varphi(x)|^{\sigma}\,dx\right)^{\frac{p}{{\sigma}}}}.
\end{equation}
Hereafter, for simplicity, we shall denote $E:=W^{2,p}(\mathbb{R}^N)$.

 We define the weighted Lebesgue space $L^r(\mathbb{R^N}, f)$ by
$$L^r(\mathbb{R^N}, f):=\left\{\varphi:\mathbb{R}^N \to \mathbb{R}: 
\varphi 
\ 
\mbox{is measurable and}
\
\int_{\mathbb{R}^N} f(x)|\varphi(x)|^r\,dx<\infty\right\},$$
and endow it  with the following norm
$$\|\varphi\|_{r, f}:=\left(\int_{\mathbb{R}^N} f(x)|\varphi(x)|^r\,dx\right)^{\frac{1}{r}}.$$
Then  $L^r(\mathbb{R^N}, f)$ is a uniformly convex Banach space. 
  {\sc Dhifli and Alsaedi} \cite{dhif}  proved that under hypothesis $(H_3)$, the embedding $E \hookrightarrow L^r(\mathbb{R}^N, f)$ is continuous and compact. Moreover, one has the following estimate
 \begin{equation}\label{esti}
\|\varphi\|^r_{r, f}\leq  S_{p^*}^{-\frac{r}{p}}|f|_{\frac{p^\ast}{p^\ast-r}}\|\varphi\|^r,\;\;\mbox{for every}\;\;\varphi\in  E.
\end{equation}

Now, let us introduce the notion of weak solutions.
\begin{definition}
A function $\varphi\in E$ is said to be a {\it weak solution} of  problem \eqref{p}, provided that 
$$\Lambda(\varphi,\psi)=\mu \int_{\mathbb{R}^N} f(x) h(\varphi) \psi\,dx+\int_{\mathbb{R}^N}|x|^{-\alpha} \varphi^{p^\ast(\alpha)-2} \varphi \psi\,dx,\;\;\mbox{for every }\;\psi\in E,$$
where $$\Lambda(\varphi,\psi):=\int_{\mathbb{R}^N}|\Delta \varphi|^{p-2}\Delta \varphi \Delta \psi-\lambda \frac{|\varphi|^{p-2}\varphi \psi}{|x|^{2p}}+|\nabla \varphi|^{p-2}\nabla \varphi \nabla \psi\;dx.$$
\end{definition}
\noindent
We define the energy functional $J_\mu: E \to \mathbb{R}$, by
$$J_\mu(\varphi):=\frac{1}{p}\|\varphi\|^p -\mu \int_{\mathbb{R}^N} f(x){ h(\varphi) \varphi}\,dx-\frac{1}{p^\ast(\alpha)}\int_{\mathbb{R}^N}|x|^{-\alpha}  \varphi^{p^\ast(\alpha)} \,dx.$$
Note that a function $\varphi\in E$ is a weak solution of problem \eqref{p}, if it   satisfies $J'_\mu(\varphi)=0$, 
i.e.,
 $\varphi$ is a critical value for $J_\mu$.
\begin{definition}
We say that  a function $\Phi\in C^{1}(F,\mathbb{R})$,  where $F$ is a Banach space,
satisfies the {\it Palais-Smale condition,} 
 if every  sequence
  $\{\varphi_{n}\}\subset F$, 
  such that
$ \Phi(\varphi_{n})$
is bounded and
$\Phi'(\varphi_{n})\rightarrow 0$
 in
 $ F^{\ast},$
 as
 $ n\rightarrow\infty,$
contains a convergent subsequence.
\end{definition}

To prove Theorem \ref{thm},  we shall need the following result which is proved in {\sc Ambrosetti and Rabinowitz} \cite[Theorem 2.4]{ambr}.
\begin{theorem}(Mountain pass theorem)\label{th1}  
  Let $\Phi\in C^{1}(F,\mathbb{R})$,  where $F$ is a Banach space, and suppose that   $ \varphi\in F$ is such that $||\varphi||>r$, for some $ r>0,$ and 
$ \inf_{||\psi||=r}\Phi(\psi)>\Phi(0)>\Phi(\varphi).$
If in addition, $\Phi$ satisfies the Palais-Smale  condition  at level $c$, then $ c$  is a critical value of $\Phi$, where 
$c:=\inf_{\gamma\in\Gamma}\max_{s\in[0,1]}\Phi(\gamma(s))$
and
$\Gamma=\{\gamma\in C([0,1], F):(\gamma(0),\gamma(1))=(0,\varphi)\}.$
\end{theorem}
\section{Proof of Theorem \ref{thm}}\label{s3}
In this section, we shall prove the first main result of this paper. More precisely, under suitable conditions, we shall prove that the functional energy associated with problem \eqref{p} satisfies the Mountain pass geometry. First, we shall prove several lemmas.
\begin{lemma}\label{coer}
Under hypotheses 
$(H_1)$
and $(H_2)$, 
there exist  $\rho>0$ and $\eta>0$ such that 
$ \|\varphi\|=\rho
 \Longrightarrow
J_\mu(\varphi)\geq \eta>0.$
\end{lemma}
\begin{proof}Let $\varphi\in E$. Invoking hypotheses
  $(H_1)$-$(H_2)$ and equation \eqref{const}, we obtain
\begin{eqnarray*}
J_\mu(\varphi)&=&\frac{1}{p}\|\varphi\|^p-\mu\int_{\mathbb{R}^N}f(x) H(\varphi)\,dx-\frac{1}{p^\ast(\alpha)}\int_{\mathbb{R}^N}|x|^{-\alpha}  u^{p^\ast(\alpha)} \,dx\\
&\geq& \frac{1}{p}\|\varphi\|^p-\frac{\mu}{r}\int_{\mathbb{R}^N}f(x)h(\varphi)\,dx-\frac{1}{p^\ast(\alpha)}S_{p^*(\alpha)}^{-\frac{p^*(\alpha)}{p}}\|\varphi\|^{p^*(\alpha)} \\
&\geq&  \frac{1}{p}\|\varphi\|^p-\frac{\mu}{r}c_1\|f\|_\infty|u|_r^r-\frac{1}{p^\ast(\alpha)}S_{p^*(\alpha)}^{-\frac{p^*(\alpha)}{p}}\|\varphi\|^{p^*(\alpha)} \\
&\geq&  \frac{1}{p}\|\varphi\|^p-\frac{\mu}{r}c_1\|f\|_\infty S_{r}^{-\frac{r}{p}}\|\varphi\|^r-\frac{1}{p^\ast(\alpha)}S_{p^*(\alpha)}^{-\frac{p^*(\alpha)}{p}}\|\varphi\|^{p^*(\alpha)}\\
&\geq&  \|\varphi\|^p\left\{\frac{1}{p}-\frac{\mu}{r}c_1\|f\|_\infty S_{r}^{-\frac{r}{p}}\|\varphi\|^{r-p}-\frac{1}{p^\ast(\alpha)}S_{p^*(\alpha)}^{-\frac{p^*(\alpha)}{p}}\|\varphi\|^{p^*(\alpha)-p}\right\},
\end{eqnarray*}
and
since  $\min(r,p^*(\alpha))>p$, we get
$$\displaystyle\lim_{\|\varphi\|\to0}\left\{\frac{1}{p}-\frac{\mu}{r}c_1\|f\|_\infty S_r^{-\frac{r}{p}}\|\varphi\|^{r-p}-\frac{1}{p^\ast(\alpha)}S_{p^*(\alpha)}^{-\frac{p^*(\alpha)}{p}}\|\varphi\|^{p^*(\alpha)-p}\right\}=\frac{1}{p}>0,$$ 
therefore, for $\rho>0$ small enough, if $ \|\varphi\|=\rho,$   we obtain
$$\eta:=\rho^p\left(\frac{1}{p}-\frac{\mu}{r}c_1\|f\|_\infty S_r^{-\frac{r}{p}}\rho^{r-p}-\frac{1}{p^\ast(\alpha)}S_{p^*(\alpha)}^{-\frac{p^*(\alpha)}{p}}\rho^{p^*(\alpha)-p}\right)>0,$$
thus
$\|\varphi\|=\rho \Longrightarrow J_\mu(\varphi)\geq \eta>0.$
This completes the proof of Lemma \ref{coer}.
\end{proof}
\begin{lemma}\label{lem2.3}Under  hypotheses of Lemma \ref{coer}, there exists $e\in E$ such that 
$\|e\|>\rho$
\
and 
\
$J_\mu(e)<0.$
\end{lemma}
\begin{proof}
Let $\varphi$ be a positive function in $\mathcal{C}_c^\infty(E)$. Then for every $s>0$ we have 
\begin{eqnarray*}
 J_\mu(s\varphi)&=&\frac{s^p}{p}\|\varphi\|^p-\mu\int_{\mathbb{R}^N}f(x) H(s\varphi)\,dx-\frac{s^{p^\ast(\alpha)}}{p^\ast(\alpha)}  
                   \int_{\mathbb{R}^N}|x|^{-\alpha}\varphi^{p^\ast(\alpha)} \,dx\\
                &\le&\frac{s^p}{p}\|\varphi\|^p-\frac{s^{p^\ast(\alpha)}}{p^\ast(\alpha)}  
                   \int_{\mathbb{R}^N}|x|^{-\alpha}\varphi^{p^\ast(\alpha)} \,dx.
\end{eqnarray*}
Since $p<p^\ast(\alpha),$ it follows that  $J_\mu(s\varphi)\rightarrow-\infty,$ as $s\rightarrow\infty$. Therefore there exists $s_0>\frac{\rho}{\|\varphi\|}$ large enough, such that  $J_\mu(s_0\varphi)<0$. If we now set $e=s_0\varphi,$ then $\|e\|>\rho$ and $J_\mu(e)<0$. 
This completes the proof of Lemma \ref{lem2.3}.
\end{proof}

\begin{lemma}\label{lem2.4}Under  hypotheses
 of Lemma \ref{coer}, $J_\mu$ satisfies the Palais-Smale condition.
\end{lemma}
\begin{proof}
Let $\{\varphi_{n}\}$ be a Palais-Smale sequence, which means that  $J_\mu(\varphi_n)$
  is bounded and 
  $ J'_\mu(\varphi_n)\rightarrow0, $
   as
   $ n\rightarrow\infty.$
 Therefore there exist $m_1>0$ and $m_2>0$ such that  
$J_\mu(\varphi_n)\leq m_1$
 and
 $ |J'_\mu(\varphi_n)|\le m_2.$
Letting $\theta:=\min(r,p^*(\alpha))$, we get by hypothesis
 $(H_1)$ that 
\begin{eqnarray*}
 \theta m_1+m_2&\ge&\theta J_\mu(\varphi_n)-\langle J'_\mu(\varphi_n), \varphi_n\rangle\\
               &\ge&\frac{\theta}{p}\|\varphi_n\|^p-\mu\theta\int_{\mathbb{R}^N}f(x) H(\varphi_n)\,dx-\frac{\theta}{p^\ast(\alpha)}  
                   \int_{\mathbb{R}^N}|x|^{-\alpha}|\varphi_n|^{p^\ast(\alpha)} \,dx\\
                   &&-\|\varphi_n\|^p+\mu\int_{\mathbb{R}^N}f(x) h(\varphi_n)\varphi_n\,dx+ \int_{\mathbb{R}^N}|x|^{-\alpha}\varphi_n^{p^\ast(\alpha)}\,dx\\
                   &\ge& (\frac{\theta}{p}-1)\|\varphi_n\|^p+\mu(r-\theta)\int_{\mathbb{R}^N}f(x) H(\varphi_n)\,dx\\
                   &+&
                   (1-\frac{\theta}{p^\ast(\alpha)})\int_{\mathbb{R}^N}|x|^{-\alpha}|\varphi_n|^{p^\ast(\alpha)} \,dx
                  \ge (\frac{\theta}{p}-1)\|\varphi_n\|^p,
\end{eqnarray*}
and since $\theta=\min(r,p^*(\alpha))>p$, it follows that the sequence $\{\varphi_n\}$ is bounded in $E$. Therefore $($up to a subsequence$)$  there exists $\varphi\in E$ such that
$$\left\{
\begin{array}{lll}
  \varphi_n&\rightharpoonup& \varphi     \mbox{ weakly in } E,\\
  \varphi_n&\longrightarrow& \varphi      \mbox{ strongly in } L^r(\mathbb{R}^N),\\
  \varphi_n&\longrightarrow& \varphi     \mbox{ a.e. in } \mathbb{R}^n,
\end{array}\right.$$
so, by  $(H_1)-(H_2)$ and the Dominated convergence theorem, 
\begin{equation}
 \label{o(1)}
 \lim_{n\to\infty}\int_{\mathbb{R}^N} f(x) H(\varphi_n)\,dx=\int_{\mathbb{R}^N} f(x) H(\varphi)\,dx.
\end{equation}
One can now show by a standard argument that the weak limit $u$ of $\{\varphi_n\}$ is a critical point of $J_\mu$ and thus $J'_\mu(\varphi)=0$.

Let $w_n:=\varphi_n-\varphi$. Then $w_n$ converges weakly to zero. Moreover, by {\sc Brezis and Lieb} \cite[Lemma 3]{ BrezisLieb}, we get 
$$|w_n|_{p^*(\alpha)}^{p^*(\alpha)}=|\varphi_n|_{p^*(\alpha)}^{p^*(\alpha)}-|\varphi|_{p^*(\alpha)}^{p^*(\alpha)}+ o(1),$$
therefore
$$\lim_{n\to\infty}\int_{\mathbb{R}^N}|x|^{-\alpha}|\varphi_n|^{p^\ast(\alpha)} -|x|^{-\alpha}|w_n|^{p^\ast(\alpha)}\,dx=\int_{\mathbb{R}^N}|x|^{-\alpha}|\varphi|^{p^\ast(\alpha)}\,dx,$$
and from \eqref{o(1)} we have
 $$\langle J'_\mu(\varphi_n),\varphi_n\rangle-\langle J'_\mu(\varphi), \varphi\rangle=\|w_n\|^p-\int_{\mathbb{R}^N}|x|^{-\alpha}|w_n|^{p^\ast(\alpha)}\,dx+o(1),$$
hence for $n$ large enough,  
\begin{equation*}
 \|w_n\|^p=\int_{\mathbb{R}^N}|x|^{-\alpha}|w_n|^{p^\ast(\alpha)}\,dx+o(1),
\end{equation*}
thus
\begin{equation}
 \label{l}
\lim_{n\to\infty}\|w_n\|^p=\lim_{n\to\infty}\int_{\mathbb{R}^N}|x|^{-\alpha}|w_n|^{p^\ast(\alpha)}=l\ge0.
\end{equation}
If $l>0$,  then by combining equation  \eqref{Slambda} with \eqref{l} we obtain 
\begin{equation}
 \label{lS}
 l\ge S_{p^*(\alpha)}^\frac{p}{p^*(\alpha)-p}.
\end{equation}

On the other hand, one has 
$$J_\mu(\varphi_n)-J_\mu(\varphi)=\frac1p\|w_n\|^p-\frac{1}{p^*(\alpha)}\int_{\mathbb{R}^N}|x|^{-\alpha}|w_n|^{p^\ast(\alpha)}\,dx+o(1),$$
so by letting $n$ tend to infinity, we get 
$$c-J_\mu(\varphi)=(\frac1p-\frac{1}{p^*(\alpha)})l,$$
and using the last equation and \eqref{lS} we obtain 
$$J_\mu(\varphi)+(\frac1p-\frac{1}{p^*(\alpha)})l=c<(\frac1p-\frac{1}{p^*(\alpha)})S_\lambda^\frac{p}{p^*(\alpha)-p},$$
which implies that 
\begin{equation}
 \label{l>0}
 J_\mu(\varphi)<0.
\end{equation}

However,  we have $\langle J'_\mu(\varphi), \varphi\rangle=0$, for every $\varphi\in E$. So, from $(H_2)$ we get 
\begin{eqnarray*}
 \|\varphi\|^p&=& \mu\int_{\mathbb{R}^N} f(x)h(\varphi)\varphi\,dx+\int_{\mathbb{R}^N} |x|^{-\alpha}|\varphi|^{p^*(\alpha)}\,dx.\\
        &\ge&r\mu\int_{\mathbb{R}^N} f(x)H(\varphi)\,dx+\int_{\mathbb{R}^N} |x|^{-\alpha}|\varphi|^{p^*(\alpha)}\,dx,
\end{eqnarray*}
therefore 
\begin{eqnarray*}
 J_\mu(\varphi)&= \frac1p\|\varphi\|^p-\mu\int_{\mathbb{R}^N} f(x)H(\varphi)\,dx-\frac{1}{p^*(\alpha)}\int_{\mathbb{R}^N} |x|^{-\alpha}|u|^{p^*(\alpha)}\,dx\\
         &\ge\frac1p\left(r\mu\int_{\mathbb{R}^N} f(x)H(\varphi)\,dx+\int_{\mathbb{R}^N} |x|^{-\alpha}|u|^{p^*(\alpha)}\,dx\right)\\  
              &- \mu\int_{\mathbb{R}^N} f(x)H(\varphi)\,dx-\frac{1}{p^*(\alpha)}\int_{\mathbb{R}^N} |x|^{-\alpha}|u|^{p^*(\alpha)}\,dx\\
         \ge&\mu(\frac rp-1)\int_{\mathbb{R}^N} f(x)H(\varphi)\,dx+(\frac1p-\frac{1}{p^*(\alpha)})\int_{\mathbb{R}^N} |x|^{-\alpha}|u|^{p^*(\alpha)}\,dx,
\end{eqnarray*}
and since $r\in(p,p^*)$ and $p<p^*(\alpha)$, it follows that $J_\mu(\varphi)\ge0$. This is in contradiction with  \eqref{l>0}. Since $l=0,$ we see by \eqref{l} 
that $\{\varphi_n\}$ converges strongly to $\varphi$ in $E$.
This completes the proof of Lemma \ref{lem2.4}.
\end{proof}

{\bf Proof of Theorem \ref{thm} }
By  Lemma \ref{coer},  there exist $\rho\in(0,\infty)$ and $\eta\in(0,\infty)$ such that
$\inf_{\|\varphi\|=\rho}J_\mu(\varphi)\ge\eta>0.$
On the other hand, by  Lemma \ref{lem2.3},  there exists $e\in E$ such that
$$\rho\leq\|e\|
\
\hbox{and}
\
  J_\mu(e)<0< \inf_{\|\varphi\|=\rho}J_\mu(\varphi),$$
hence, combining Lemma \ref{lem2.4} and Theorem \ref{th1},  we can establish the existence of a critical point $\varphi_\mu$.  

Moreover,  $\varphi_\mu$ is characterized by 
$$J_\mu(\varphi_\mu)=\inf_{\gamma\in\Gamma}\max_{t\in[0,1]}J_\mu(\gamma(t)),$$
 where
$$\Gamma:=\{\gamma\in C([0,1],X): (\gamma(0),\gamma(1))=(0,e)\},$$
so if we take $\gamma(s)=se$, then there exists $s_0\in[0,1]$ such that $\|s_0 e\|=\rho,$
hence invoking Lemma \ref{lem2.3}, we obtain
\begin{equation}
 \label{1sol}
 J_\mu(\varphi_\mu)\ge\eta>0.
\end{equation}
This completes the proof of Theorem \ref{thm}.\qed

\section{Proof of Theorem \ref{thmm}} \label{s4}

In this section, we shall prove Theorem \ref{thmm}. The proof 
will be divided into several lemmas.
\begin{lemma}\label{lem3.1}
Under hypotheses $(H_2)$
and
$(H_3)$, there exist  positive constants $\mu_0$, $\rho,$ and $\eta$ such that for every $\mu\in(0,\mu_0),$
$ \|\varphi\|=\rho
\Longrightarrow
J_\mu(\varphi)\geq \eta>0.$
\end{lemma}
\begin{proof}Let $\varphi\in E$. 
Invoking 
hypotheses
  $(H_2)$-$(H_3)$, equations \eqref{const},  \eqref{esti} and the H\"older inequality, we obtain
\begin{eqnarray}\label{999}
J_\mu(\varphi)&=&\frac{1}{p}\|\varphi\|^p-\mu\int_{\mathbb{R}^N}f(x) H(\varphi)\,dx-\frac{1}{p^\ast(\alpha)}\int_{\mathbb{R}^N}|x|^{-\alpha}  \varphi^{p^\ast(\alpha)} \,dx\nonumber\\
&\geq& \frac{1}{p}\|\varphi\|^p-\frac{\mu}{r}\int_{\mathbb{R}^N}f(x)h(\varphi)\,dx-\frac{1}{p^\ast(\alpha)}S_{p^\ast}^{-\frac{p^*(\alpha)}{p}}\|\varphi\|^{p^*(\alpha)} \nonumber\\
&\geq&  \frac{1}{p}\|\varphi\|^p-\frac{\mu}{r}c_1\|\varphi\|^r_{r,f}-\frac{1}{p^\ast(\alpha)}S_{p^\ast}^{-\frac{p^*(\alpha)}{p}}\|\varphi\|^{p^*(\alpha)} \\
&\geq&  \frac{1}{p}\|\varphi\|^p-\frac{\mu}{r}c_1\|f\|_\frac{p^*}{p^*-r}S_{p^\ast}^{-\frac{r}{p}}\|\varphi\|^r-\frac{1}{p^\ast(\alpha)}S_{p^\ast}^{-\frac{p^*(\alpha)}{p}}\|\varphi\|^{p^*(\alpha)}\nonumber\\
&\geq&  \|\varphi\|^r\left(\frac{1}{p}\|\varphi\|^{p-r}-\frac{\mu}{r}c_1\|f\|_\frac{p^*}{p^*-r}S_{p^\ast}^{-\frac{r}{p}}-\frac{1}{p^\ast(\alpha)}S_{p^\ast}^{-\frac{p^*(\alpha)}{p}}\|\varphi\|^{p^*(\alpha)-r}\right)\nonumber\\
&\geq&\|\varphi\|^r\left(h\bigl(\|\varphi\|\bigr)-\frac{\mu}{r}c_1\|f\|_\frac{p^*}{p^*-r}S_{p^\ast}^{-\frac{r}{p}}\right),\nonumber
\end{eqnarray}
where $$h(s):=\frac{1}{p}s^{p-r}-\frac{1}{p^\ast(\alpha)}S_{p^\ast}^{-\frac{p^*(\alpha)}{p}}s^{p^*(\alpha)-r}.$$

It is not difficult to prove that $h$ attains its global maximum at
$$s_0:=\left(\frac{p^*(\alpha)(p-r)S_{p^\ast}^{\frac{p^*(\alpha)}{p}}}{p\bigl(p^*(\alpha)-r\bigr)} \right)^\frac{1}{p^*-p}.$$
Set
\begin{equation}
 \label{mu0}
 \mu_0:=\frac{rf(s_0)}{c_1\|f\|_\frac{p^*}{p^*-r}S_{p^\ast}^{-\frac{r}{p}}}.
\end{equation}
Then for every $\mu\in(0,\mu_0)$, we have $$h(s_0)-\frac{\mu}{r}c_1\|f\|_\frac{p^*}{p^*-r}S_{p^\ast}^{-\frac{r}{p}}>0,$$
and since $h$ is continuous,  we can find $\rho>0$ such that $$h(\rho)-\frac{\mu}{r}c_1\|f\|_\frac{p^*}{p^*-r}S_{p^\ast}^{-\frac{r}{p}}>0,$$
thus for every $\varphi\in E$ with $\|\varphi\|=\rho$,  we have
$$J_\mu(\varphi)\ge\eta:=\rho^r\left(h(\rho)-\frac{\mu}{r}c_1\|f\|_\frac{p^*}{p^*-r}S_{p^\ast}^{-\frac{r}{p}}\right)>0.$$
This completes the proof of Lemma \ref{lem3.1}.
\end{proof}
\begin{lemma}\label{lem3.2}
There exists $e\in E$ 
such that $\|e\|>\rho$ and $J_\mu(e)<0$.
\end{lemma}
\begin{proof}
Since the proof is very similar to the proof of Lemma \ref{lem2.3},  we shall omit it. 
\end{proof}
\begin{lemma}\label{lem3.3}
Under hypotheses $(H_2)$
and
$(H_3)$, the functional $J_\mu$ satisfies the Palais-Smale condition.
\end{lemma}
\begin{proof}
Let $\{\varphi_{n}\}$ be a Palais-Smale sequence.  By the argument from the previous
 section, it follows that there exist $m_1>0$ and $m_2>0$,  such that  
$J_\mu(\varphi_n)\leq m_1$
 and
 $ |J'_\mu(\varphi_n)|\le m_2.$
 
Let us prove that $\{\varphi_{n}\}$ is bounded. If not, then up to a subsequence we can assume that $\|\varphi_n\|\to\infty,$ as $n\to \infty$.  By virtue of hypotheses
 $(H_2)-(H_3),$ we get
\begin{eqnarray*}
p^\ast(\alpha) m_1+m_2&\ge&p^\ast(\alpha) J_\mu(\varphi_n)-\langle J'_\mu(\varphi_n), \varphi_n\rangle\\
               &\ge&\frac{p^\ast(\alpha)}{p}\|\varphi_n\|^p-\mu p^\ast(\alpha)\int_{\mathbb{R}^N}f(x) H(\varphi_n)\,dx\\
               &&-\frac{p^\ast(\alpha)}{p^\ast(\alpha)}  
                   \int_{\mathbb{R}^N}|x|^{-\alpha}\varphi_n^{p^\ast(\alpha)} \,dx-\|\varphi_n\|^p\\
                   &&+\mu\int_{\mathbb{R}^N}f(x) h(\varphi_n)\varphi_n\,dx+ \int_{\mathbb{R}^N}|x|^{-\alpha}\varphi_n^{p^\ast(\alpha)}\,dx\\
                   &\ge&\bigl (\frac{p^\ast(\alpha)}{p}-1\bigr)\|\varphi_n\|^p+\mu(r-p^\ast(\alpha))\int_{\mathbb{R}^N}f(x) H(\varphi_n)\,dx\\
                  &\ge& \bigl(\frac{p^\ast(\alpha)}{p}-1\bigr)\|\varphi_n\|^p-\mu(p^\ast(\alpha)-r)c_1\|f\|_\frac{p^*}{p^*-r}S_{p^\ast}^{-\frac{r}{p}}\|\varphi_n\|^r.
\end{eqnarray*}
Since $r<p$,   a contradiction is obtained  by letting $n$ in the last inequality
tend
 to infinity, therefore $\{\varphi_{n}\}$ is indeed bounded. 
The rest of the proof is analogous to the proof of  Lemma \ref{lem2.4}.
This completes the proof of Lemma \ref{lem3.3}.
\end{proof}
{\bf Proof of Theorem \ref{thmm}.}  Let $\mu\in(0,\mu_0)$,  where $\mu_0$ is defined in \eqref{mu0}. Combining Lemmas \ref{lem3.1}, \ref{lem3.2}, and \ref{lem3.3} with  Theorem \ref{th1}, we can deduce that problem \eqref{p} has a weak solution $\psi_\mu$ as a
 critical point for $J_\mu$. Moreover, as in the proof of \eqref{1sol}, one has
\begin{equation}\label{mn}J_\mu(\psi_\mu)\ge\eta>0.\end{equation}
Now, by  Lemma \ref{lem3.1}, we can see that 
$\inf_{\psi\in \partial B(0,\rho)} J_\mu(\psi) >0.$
Moreover,  by  Lemma \ref{lem3.2}, and equation \eqref{999},  we get
$$-\infty <\underline{c}:=\inf_{\psi\in \overline{B(0,\rho)}}(J_\mu(\psi) )<0.$$

Let   $\varepsilon>0$  be  such that
\begin{equation}\label{epsil}0<\varepsilon <\inf_{\psi\in \partial B(0,\rho)} J_\mu(\psi) -\inf_{\psi\in B(0,\rho)}J_\mu(\psi).\end{equation}
If we consider the functional  $J_\mu :\overline{B(0,\rho)}\rightarrow \mathbb{R}$, then by the  Ekeland variational principle  there exists  $\psi_{\varepsilon }\in
\overline{B(0,\rho)}$, such that
\begin{equation}
\left\{
\begin{array}{ll}
\underline{c}\leq J_\mu (\psi_{\varepsilon })\leq \underline{c}+\varepsilon  
\\
J_\mu (\psi_{\varepsilon })<J_\mu (\psi)+\varepsilon ||\psi-\psi_{\varepsilon }||, \;\psi\neq
\psi_{\varepsilon},
\end{array}
\right.
\end{equation}
so by \eqref{epsil}, we have
\begin{equation}
J_\mu (\psi_{\varepsilon })\leq \inf_{\psi\in \overline{B(0,\rho)}} J_\mu(\psi)+\varepsilon \leq\inf_{\psi\in B(0,\rho)}J_\mu(\psi) +\varepsilon <\inf_{\psi\in\partial B(0,\rho)} J_\mu(\psi),
\end{equation}
which  implies that $\psi_{\varepsilon }\in B(0,\rho)$.

On the other hand, if we define the functional $\Phi_\mu:\overline{B(0,\rho)}\rightarrow \mathbb{R}$ by
$\Phi_\mu(\psi):=J_\mu (\psi)+\varepsilon \|\psi-\psi_{\varepsilon }\|,$
then  $\psi_{\varepsilon }$ is a global minimum of $\Phi_\mu$. Therefore,  for $s\in (0,1)$ small enough,  we have
$$\frac{\Phi_\mu(\psi_{\varepsilon }+s \psi)-\Phi_\mu(\psi_{\varepsilon})}{s}\geq 0,\
\hbox{for every}\
 \psi\in B(0,1),$$
i.e.,
$$\frac{J_\mu (\psi_{\varepsilon }+s \psi)-J_\mu (\psi_{\varepsilon })}{s}+\varepsilon
\left\Vert \psi\right\Vert \geq 0.$$
By letting $s$ tend to zero, we get
$
\langle J'_\mu(\psi_{\varepsilon }), \psi\rangle+\varepsilon \|\psi\|
\geq 0.$
This implies that
$\|J'_\mu(\psi_{\varepsilon })\| \leq \varepsilon .$

If we put $w_{n}:=\psi_{\frac{1}{n}}$,  we obtain  $\left\{ w_{n}\right\} \subset B(0,\rho)$. Moreover,
$J_\mu (w_{n})\rightarrow \underline{c}<0,$
 and
 $J'_\mu(w_{n})\rightarrow 0,$
 as
 $n\to \infty.$
Since  $
\left\{ w_{n}\right\} \subset B(0,\rho)$, it follows that  $\left\{ w_{n}\right\}$  is bounded in $E$. So, up to a subsequence still denoted by $w_{n}$, there exists $\psi_\mu\in E$, such that  $\left\{w_{n}\right\} $ converges weakly to $\psi_\mu\in E$. Invoking  Lemma \ref{lem3.3}, we see  that  $w_{n}\rightarrow \psi_\mu$ strongly in $E.$

Now, the fact that  $J_\mu \in C^{1}(E,\mathbb{R})$, implies that
$J'_\mu(w_{n})\rightarrow J_\mu(\psi_\mu),$
as
$n\rightarrow \infty,$
so we have
\begin{equation}\label{nm}
J'_\mu(\psi_\mu)=0
\
\text{and }J_\mu (\psi_\mu)<0,
\end{equation}
hence  $\psi_\mu$ is  a nontrivial weak solution of
problem \eqref{p}. Moreover,  by combining equation \eqref{mn} with equation \eqref{nm}, we obtain that $J_\mu (\varphi_\mu)<0<J_\mu (\psi_\mu),$
i.e., $u_\mu$ and $\psi_\mu$ are distinct. This completes the proof of Theorem \ref{thmm}.\qed
\section{An Application}\label{s5}
As an application of our  results, we shall consider the following problem 
\begin{equation}\label{qq}
\Delta_{p}^{2}\varphi-\lambda \frac{|\varphi|^{p-2}\varphi}{|x|^{2p}}+\Delta_p \varphi=  \mu f(x)|\varphi|^{r-2}\varphi+ \frac{|\varphi|^{p^\ast(p)-2}\varphi}{|x|^{p}} \quad \mbox{ in }\mathbb{R}^N,
\end{equation}
where $1<p<\frac{N}{2}$ and $\lambda>0$. 
We note that problems of type \eqref{qq}  describe e.g.,
 the deformations of an elastic beam. Also, they
 give a model for considering traveling waves in suspension bridges.
 
It is not difficult to see that 
$1<\alpha=p<2p$
and 
$
h(\varphi)=|\varphi|^{r-2}\varphi
$
satisfies the second inequality of 
hypotheses
$(H_1)$
and
$(H_3),$
 with
$c_1=1>0.$ Moreover, a simple calculation shows that $H(\varphi)=\frac{1}{r}|\varphi|^r$ which satisfies $rH(\varphi)=h(\varphi)\varphi,$
 so hypothesis $(H_2)$ is also satisfied for every $\sigma>0$. 
 
Hence if  $r\in(p,p^\ast)$ and   $f \in L^\infty(\mathbb{R}^N),$
then
Theorem \ref{thm} implies that  for every
 $\mu>0,$ there exists $\lambda_0>0$ such that for every $\lambda\in (0,\lambda_0),$
  problem \eqref{qq} has
  a nontrivial solution. Moreover, if $1<r<p$  and 
 $$0<f \in L^{\frac{p^\ast}{p^\ast-r}}(\mathbb{R}^N)\cap L^s_{loc}(\mathbb{R}^N),
 \
 \hbox{  for some}
 \
 s \in (\frac{p^\ast}{p^\ast-r}, \frac{p}{p-r}),$$
  then Theorem \ref{thmm} implies the existence of $\lambda_0>0$ and $\mu_0>0$ such that for every $\lambda\in (0,\lambda_0)$ and $\mu\in (0,\mu_0)$, problem \eqref{qq} has
   at least two nontrivial solutions.

\section{Conclusion}\label{s6}
 The variational method has a long and rich history, and it has given rise to the functional energy. The Mountain pass theorem is used in the first part of this paper to prove the existence of a nontrivial solution for a p-biharmonic problem involving the Hardy-Sobolev exponent.  
 Our first main result generalizes the paper of  {\sc Ghoussoub and Yuan} \cite{8}.
 
In the second part of the paper, the Mountain pass theorem is combined with the Ekeland variational principle to prove the existence of two nontrivial solutions. 
Our second main result of this paper generalizes the work of {\sc Perrera and Zou} \cite{23}. 

We note that the manipulation of the critical Hardy nonlinearity is more complicated and the improvement method used here is an application of the Brezis-Lieb lemma.
As the foundation for further improvements, we aim to obtain even stronger results for problems with discontinuous nonlinearities.

\subsection*{Acknowledgments}
 Repov\v{s} was supported by the Slovenian Research Agency grants P1-0292, J1-4031, J1-4001, N1-0278, N1-0114, and N1-0083. We thank the referee for comments and suggestions.

\end{document}